\documentclass[aap,MSNbibl,citesort,dvips]{arximspdf}
\usepackage{graphicx}

\doi{10.1214/10-AAP746}
\volume{21}
\issue{6}
\pubyear{2011}
\firstpage{2171}
\lastpage{2190}

\makeatletter
\newcommand{\overlineh}{\overline}
\newproclaim{remark}{Remark}
\newtheorem{theorem}{Theorem}

\def\r{{\mathbb R}}
\def\d{{\mathrm d}}
\def\i{{\mathrm i}}
\newcommand{\p}{\mathbb{P}}
\newcommand{\e}{\mathbb{E}}
\newcommand{\R}{\mathbb{R}}
\renewcommand{\epsilon}{\varepsilon}
\makeatother

\begin{document}
\begin{frontmatter}

\title{A Wiener--Hopf Monte Carlo simulation technique for L\'evy processes}
\runtitle{A Wiener--Hopf Monte Carlo simulation technique}

\begin{aug}
\author[A]{\fnms{A.} \snm{Kuznetsov}\thanksref{th3}\ead[label=e1]{kuznetsov@mathstat.yorku.ca}},
\author[B]{\fnms{A. E.} \snm{Kyprianou}\corref{}\thanksref{th1,th2}\ead[label=e2]{a.kyprianou@bath.ac.uk}},
\author[C]{\fnms{J. C.} \snm{Pardo}\thanksref{th1}\ead[label=e3]{jcpardo@cimat.mx}}\\
\and
\author[B]{\fnms{K.} \snm{van Schaik}\thanksref{th2}\ead[label=e4]{k.van.schaik@bath.ac.uk}}
\runauthor{Kuznetsov, Kyprianou, Pardo and van Schaik}
\affiliation{York University, University of Bath, Centro de
Investigaci\'on en Matem\'aticas and~University of Bath}
\address[A]{A. Kuznetsov\\
Department of Mathematics and Statistics\\
York University\\
Toronto, Ontario, M3J 1P3\\
Canada\\
\printead{e1}} 
\address[B]{A. E. Kyprianou\\
K. van Schaik\\
Department of Mathematical Sciences\\
University of Bath\\
Claverton Down, Bath, BA2 7AY\\
United Kingdom\\
\printead{e2}\\
\hphantom{\textsc{E-mail:} }\printead*{e4}}
\address[C]{J. C. Pardo\\
Centro de Investigaci\'on en Matem\'aticas\\
A.C. Calle Jalisco s/n.\\
36240 Guanajuato\\
M\'exico\\
\printead{e3}}
\end{aug}
\thankstext{th3}{Supported by the
Natural Sciences and Engineering Research Council of Canada.}
\thankstext{th1}{Supported by
EPSRC Grant  EP/D045460/1.}
\thankstext{th2}{Supported by the AXA Research Fund.}

\received{\smonth{12} \syear{2009}}
\revised{\smonth{10} \syear{2010}}

%
\begin{abstract}
We develop a completely new and straightforward method for simulating
the joint
law of the position and running maximum at a~fixed time of a general L\'
evy process
with a view to application in insurance and financial mathematics.
Although different,
our method takes lessons from Carr's so-called ``Canadization'' technique
as well as Doney's
method of stochastic bounds for L\'evy processes; see Carr  [\textit{Rev. Fin. Studies}
\textbf{11} (1998) 597--626]
and Doney  [\textit{Ann. Probab.}
\textbf{32} (2004) 1545--1552]. We rely fundamentally on the Wiener--Hopf
decomposition for L\'evy processes as well as taking advantage of
recent developments in factorization techniques of the latter theory
due to Vigon  [Simplifiez vos L\'evy en titillant la
factorization de
Wiener--Hopf (2002) Laboratoire de Math\'ematiques de L'INSA de Rouen] and Kuznetsov [\textit{Ann. Appl. Probab.}
\textbf{20} (2010) 1801--1830]. We illustrate our
Wiener--Hopf Monte Carlo method on a number of different processes,
including a new family of L\'evy processes called hypergeometric L\'evy
processes. Moreover, we illustrate the robustness of working with a
Wiener--Hopf decomposition with two extensions. The first extension
shows that if one can successfully simulate for a given L\'evy
processes then one can successfully simulate for any independent sum of
the latter process and a compound Poisson process. The second extension
illustrates how one may produce a
straightforward approximation for simulating the two-sided exit problem.
\end{abstract}

%
\begin{keyword}[class=AMS]
\kwd{65C05}
\kwd{68U20}.
\end{keyword}
\begin{keyword}
\kwd{L\'evy processes}
\kwd{exotic option pricing}
\kwd{Wiener--Hopf factorization}.
\end{keyword}

\vspace*{-6pt}
\end{frontmatter}

\section{Introduction}
Let us suppose that $X=\{X_{t}\dvtx t\geq0\}$ is a general L\'{e}vy process
with law $\mathbb{P}$ and L\'evy measure $\Pi$.
That is to say, $X$ is a Markov process with paths that are right
continuous with left limits such that\vadjust{\goodbreak} the increments are stationary
and independent and whose characteristic function at each time $t$ is
given by the L\'evy--Khinchine representation
\begin{equation}\label{eq1}
\mathbb{E}[e^{\mathrm{i}\theta X_t}]=e^{-t\Psi(\theta)} ,\qquad\theta
\in\mathbb{R} ,
\end{equation}
where
%
\begin{equation}\label{lrep}
\Psi(\theta) =
\mathrm{i}\theta a + \frac{1}{2}\sigma^2\theta^2+
\int_{\r}\bigl(1-e^{\mathrm{i}\theta x}+
\mathrm{i}\theta x \mathbf{1}_{\{|x|<1\}}\bigr)\Pi(\mathrm{d}x) .
\end{equation}
We have $a\in\mathbb{R}$, $\sigma^2\ge0$ and
$\Pi$ is a measure supported on $\mathbb{R}$
with $\Pi(\{0\})= 0$ and
$\int_{\mathbb{R}}(x^2\wedge1) \Pi(\mathrm{d}x)<\infty$.
Starting with the early work of Madan and Seneta~\cite{MD}, L\'evy
processes have played a central role in the theory of financial
mathematics and statistics (see, e.g., the books \cite{BL,CT,S,SC}). More recently, they have been extensively used in modern
insurance risk theory
(see, e.g., Kl\"uppelberg, Kyprianou and Maller \cite{KKM}, Song and Vondra\v
{c}ek~\cite{SV}). The basic idea in financial mathematics and
statistics is that the logarithm of the stock price or risky asset follows the
dynamics of a L\'evy process whilst in insurance mathematics, it is the
L\'evy process itself which models the surplus wealth of an insurance
company until ruin. There are also extensive applications of L\'evy
processes in queuing theory, genetics and mathematical biology as well
as through their appearance in the theory of stochastic differential equations.

In both financial and insurance settings, a key quantity of generic
interest is the joint law of the current position and the running
maximum of a L\'evy process at a fixed time if not the individual
marginals associated with the latter bivarite law. Consider the
following example. If we define $\overline{X}_t = \sup_{s\leq t}X_s$,
then the pricing of barrier options boils down to evaluating
expectations of the form $
\mathbb{E}[f(x+X_t) \mathbf{1}_{\{x+\overline{X}_t >b\}}]
$
for some appropriate function $f(x)$ and threshold $b>0$. Indeed if
$f(x) = (K-e^x)^+$ then the latter expectation is related to the value
of an ``up-and-in'' put. In credit risk, one is predominantly interested
in the quantity $\widehat{\mathbb{P}}(\overline{X}_t <x)$ as a
function in $x$ and~$t$, where $\widehat{\mathbb{P}}$ is the law of
the dual process $-X$. Indeed it is as a functional of the latter
probabilities that the price of a credit default swap is computed; see,
for example, the recent book of Schoutens and Cariboni \cite{SC}. One
is similarly interested in $\widehat{\mathbb{P}}(\overline{X}_t \geq
x)$ in ruin theory as these probabilities are also equivalent to the
finite-time ruin probabilities.

One obvious way to do Monte Carlo simulation of expectations involving
the joint law of $(X_t, \overline{X}_t)$ that takes advantage of the
stationary and independent increments of L\'evy processes is to take a
random walk approximation to the L\'evy process, simulate multiple
paths, taking care to record the maximum for each run. When one is able
to set things up in this way so that one samples exactly from the
distribution of $X_t$, the law of the maximum of the underlying random
walk will not agree with the law of $\overline{X}_t$.

Taking account of the fact that all L\'evy processes respect a
fundamental path decomposition known as the Wiener--Hopf\vadjust{\goodbreak} factorization,
it turns out there is another very straightforward way to perform
Monte Carlo simulations for expectations involving the joint law of
$(X_t, \overline{X}_t)$ which we introduce in this paper. Our method
allows for exact sampling from the law of $(X_{\mathbf{g}}, \overline
{X}_{\mathbf{g}})$ where $\mathbf g$ is
a random time whose distribution can be concentrated arbitrarily close
around $t$.

There are several advantages of the technique. First, when it is
taken in context with very recent developments in Wiener--Hopf theory
for L\'evy processes, for example, recent advances in the theory of
scale functions for spectrally negative processes (see Kyprianou, Pardo
and Rivero
\cite{KPR}), new complex analytical techniques due to Kuznetsov \cite
{Kuz} and Vigon's theory of philanthropy (see \cite{V}), one may
quickly progress the algorithm to quite straightforward numerical work.
Second, our Wiener--Hopf method takes advantage of a similar feature
found in the, now classical, ``Canadization'' method of Carr \cite{Carr}
for numerical evaluation of optimal stopping problems. The latter is
generally acknowledged as being more efficient than appealing to
classical random walk approximation Monte Carlo methods. Indeed, later
in this paper, we present our numerical findings with some indication
of performance against the method of random walk approximation. In this
case, our Wiener--Hopf method appears to be extremely effective.
Third, in principle, our method handles better the phenomena of
discontinuities which can occur with functionals of the form $
\mathbb{E}[f(x+X_t) \mathbf{1}_{\{x+\overline{X}_t >b\}}]
$
at the boundary point $x=b$. It is now well understood that the issue
of regularity of the upper and lower half line for the underlying L\'
evy process (see Chapter~6 of~\cite{Kypbook} for a definition) is
responsible the appearance of a discontinuity at \mbox{$x=b$} in such
functions (cf. \cite{AK}). The nature of our Wiener--Hopf method
naturally builds the distributional atom which is responsible for this
discontinuity into the simulations.

Additional advantages to the method we propose include its simplicity
with regard to numerical implementation. Moreover, as we shall also see
in Section \ref{Sect:extensions} of this paper, the natural
probabilistic structure that lies behind our so-called Wiener--Hopf
Monte Carlo method also
allows for additional creativity when addressing some of the
deficiencies of the method itself.

\section{Wiener--Hopf Monte Carlo simulation technique}

The basis of the algorithm is the following simple observation which
was pioneered by Carr~\cite{Carr} and subsequently used in several
contexts within mathematical finance for producing approximate
solutions to free boundary value problems that appear as a result of
optimal stopping problems characterizing the value of an American-type option.

Suppose that ${\mathbf{e}_1, \mathbf{e}_2, \ldots }$ are a sequence
of i.i.d. exponentially distributed random variables with unit mean.
Suppose they are all defined on a common product space with product law
$\mathbf{P}$ which is orthogonal to the probability space on which the
L\'evy process $X$ is defined. For all $t>0$, we\vadjust{\goodbreak} know from the Strong
Law of Large Numbers that
%
\begin{equation}
\sum_{i=1}^n \frac{t}{n}\mathbf{e}_i \rightarrow t  \qquad \mbox{as }n\uparrow\infty
\label{SLLN}
\end{equation}
$\mathbf{P}$-almost surely.
The random variable on the left-hand side above is equal in law to a
Gamma random variable with parameters $n$ and $n/t$. Henceforth, we
write it $\mathbf{g}(n, n/t)$. Recall that $\mathbb{P}$ is our notation
for the law of the L\'evy process~$X$. Then writing $\overline{X}_t =
\sup_{s\leq t}X_s$ 
we argue the case that, for sufficiently large~$n$, a suitable
approximation to $\mathbb{P}(X_t\in\mathrm{d}x,   \overline{X}_t \in
\mathrm{d}y)$ is $(\mathbf{P}\times\mathbb{P})(X_{\mathbf{g}(n,n/t)}
\in\mathrm{d}x,   \overline{X}_{\mathbf{g}(n,n/t)} \in\mathrm{d}y)$.

This approximation gains practical value in the context of Monte Carlo
simulation when we take advantage of the fundamental path decomposition
that applies to all L\'evy processes over exponential time periods
known as the Wiener--Hopf factorization.\vspace*{-3pt}


\begin{theorem}\label{thm_main}
For all $n\ge1$ and $\lambda>0$, define $\mathbf{g}(n, \lambda) : =
\sum_{i=1}^n\mathbf{e}_i/\lambda$. Then
\begin{equation}\label{distr_identity}
\bigl(X_{\mathbf{g}(n,\lambda)}, \overline{X}_{\mathbf{g}(n,\lambda)}\bigr)
\stackrel{d}{=}(V(n,\lambda), J(n,\lambda)),
\end{equation}
where $V(n,\lambda)$ and $J(n,\lambda)$ are defined iteratively for
$n\ge1$ as
\begin{eqnarray*}
V(n,\lambda) &=& V(n-1,\lambda)+S^{(n)}_\lambda+ I^{(n)}_\lambda,\\[-2pt]
J(n, \lambda) &=& \max \bigl( J(n-1,\lambda) , V(n-1,\lambda)+
S^{(n)}_\lambda \bigr)
\end{eqnarray*}
and $V(0,\lambda)=J(0,\lambda)=0$.
Here, $S^{(0)}_\lambda= I^{(0)}_\lambda = 0$, $\{S^{(j)}_\lambda\dvtx
j\ge1\}$ are an i.i.d. sequence of random variables with common
distribution equal to that of $\overline{X}_{\mathbf{e}_1/\lambda}$
and $\{I^{(j)}_\lambda\dvtx  j\ge1\}$ are another i.i.d. sequence of
random variables with common distribution equal to that of $\underline
{X}_{\mathbf{e}_1/\lambda}$.\vspace*{-3pt}
\end{theorem}

\begin{pf}
The Wiener--Hopf factorization tells us that $\overline{X}_{\mathbf
{e}_1/\lambda}$ and $X_{\mathbf{e}_1/\lambda} -\overline
{X}_{\mathbf{e}_1/\lambda}$ are independent and the second of the
pair is equal in distribution to~$\underline{X}_{\mathbf{e}_1/\lambda
}$. This will constitute the key element of the proof.

Fix $n\geq1$. Suppose we define $\overline{X}_{s,t} = \sup_{s\leq
u\leq t}X_u$. Then it is trivial to note that
%
\begin{eqnarray}
\label{max}
&&\bigl(X_{\mathbf{g}(n,\lambda)},
\overline{X}_{\mathbf{g}(n,\lambda)}\bigr)\nonumber\hspace*{-15pt}
\\[-9pt]
\\[-9pt]
&& \qquad =  \bigl(X_{\mathbf{g}(n-1,\lambda)} +\bigl(X_{\mathbf{g}(n,\lambda)}-
X_{\mathbf{g}(n-1,\lambda)}\bigr), \overline{X}_{\mathbf{g}(n-1,\lambda
)}\vee\overline{X}_{\mathbf{g}(n-1, \lambda), \mathbf{g}(n,
\lambda)} \bigr),
\nonumber\hspace*{-15pt}
\end{eqnarray}
where $\mathbf{g}(0, \lambda) := 0$. If we define $X^{(n)}_t =
X_{\mathbf{g}(n-1, \lambda)+t} - X_{\mathbf{g}(n-1, \lambda)}$ and
$\overline{X}^{(n)}_{\mathbf{e}_n/\lambda} = \sup_{s\leq\mathbf
{e}_n/\lambda}X^{(n)}_s$, then from (\ref{max}) it follows that
\[
\bigl(X_{\mathbf{g}(n,\lambda)}, \overline{X}_{\mathbf{g}(n,\lambda)}\bigr)
=  \bigl(X_{\mathbf{g}(n-1,\lambda)} +X^{(n)}_{\mathbf{e}_n/\lambda
}, \overline{X}_{\mathbf{g}(n-1,\lambda)}\vee\bigl(X_{\mathbf
{g}(n-1,\lambda)} + \overline{X}^{(n)}_{\mathbf{e}_n/\lambda
}\bigr) \bigr).
\]
Now noting that the process $X^{(n)}$ is independent of $\{X_s \dvtx  s\leq
\mathbf{g}(n-1,\lambda)\}$ and has law $\mathbb{P}$ and,
moreover,\vadjust{\goodbreak}
recalling the distributional Wiener--Hopf decomposition described at the
beginning of the proof, it follows that
\[
\bigl(X_{\mathbf{g}(n,\lambda)}, \overline{X}_{\mathbf{g}(n,\lambda)}\bigr)
\stackrel{d}{=} \bigl(X_{\mathbf{g}(n-1,\lambda)} + S^{(n)}_\lambda
+ I^{(n)}_\lambda, \overline{X}_{\mathbf{g}(n-1,\lambda)}\vee
\bigl(X_{\mathbf{g}(n-1,\lambda)} + S^{(n)}_\lambda \bigr)\bigr),
\]
where $S^{(n)}_\lambda$ and $I^{(n)}_\lambda$ defined as in the
statement of the theorem. The conclusion of the theorem now follows immediately.
%
%
\end{pf}

Note that the idea of embedding a random walk into the path of a L\'evy
process with two types of step distribution determined by the
Wiener--Hopf factorization has been used in a different, and more
theoretical context by Doney \cite{Don04}.

Given (\ref{SLLN}), it is clear that the pair $(V(n, n/t), J(n, n/t))$
converges in distribution to $(X_t, \overline{X}_t)$. This suggests
that we need only to be able to simulate i.i.d. copies of the
distributions of $S_{n/t}: = S^{(1)}_{n/t}$ and $I_{n/t} : =
I^{(1)}_{n/t}$ and then by a simple functional transformation we may
produce a realisation of the random variables $(X_{\mathbf{g}(n,
n/t)},\overline{X}_{\mathbf{g}(n, n/t)})$. Given a suitably nice
function $F$, using standard Monte Carlo methods one estimates for
large $k$
%
\begin{equation}
\mathbb{E}[F(X_t, \overline{X}_{t })] \simeq\frac{1}{k}\sum
_{m=1}^k F\bigl(V^{(m)}(n, n/t), J^{(m)}(n, n/t)\bigr),
\label{WH-MC}
\end{equation}
where $(V^{(m)}(n, n/t), J^{(m)}(n, n/t))$ are i.i.d. copies of
$(V(n,n/t), J(n, n/t))$. Indeed the strong law of large numbers implies
that the right-hand side above converges almost surely as $k\uparrow
\infty$ to $\mathbf{E}\times\mathbb{E}(F(X_{\mathbf{g}(n,n/t)} ,
\overline{X}_{\mathbf{g}(n,n/t)} ) )$ which in turn converges as
$n\uparrow\infty$ to $\mathbb{E}(F(X_t, \overline{X}_{t }))$.


\section{Implementation}\label{sec_impl}

The algorithm described in the previous section only has practical
value if one is able to sample from the distributions of $\overline
{X}_{\mathbf{e}_1/\lambda}$ and $-\underline{X}_{\mathbf
{e}_1/\lambda}$. It would seem that this, in itself, is not that much
different from the problem that it purports to
solve. However, it turns out that there are many tractable examples and
in all cases this is due to the tractability of their Wiener--Hopf
factorizations.

Whilst several concrete cases can be handled from the class of
spectrally one-sided L\'evy processes thanks to recent development in
the theory of scale functions, which can be used to described the laws
of $\overline{X}_{\mathbf{e}_1/\lambda}$ and $-\underline
{X}_{\mathbf{e}_1/\lambda}$ (cf. \cite{HK,KR}), we give here two
large families of two-sided jumping L\'evy processes that have
pertinence to mathematical finance to show how the algorithm may be implemented.

\subsection{\texorpdfstring{$\beta$-class of L\'evy processes}{beta-class of L\'evy processes}}

The $\beta$-class of L\'evy processes, introduced in \cite{Kuz}, is a
10-parameter L\'evy process which has characteristic exponent
\begin{eqnarray*}
\Psi(\theta) &=& \mathrm{i}a\theta+ \frac{1}{2}\sigma^2 \theta^2 +
\frac{c_1}{\beta_1}  \biggl\{ {\mathrm{B}}(\alpha_1, 1-\lambda_1)-
{\mathrm{B}} \biggl( \alpha_1 - \frac{\mathrm{i}\theta}{\beta_1},
1- \lambda_1 \biggr) \biggr\}\\
&&{} + \frac{c_2}{\beta_2} \biggl\{ {\mathrm{B}}(\alpha_2,
1-\lambda_2) - {\mathrm{B}} \biggl( \alpha_2 + \frac{\mathrm{i}\theta}{\beta_2}, 1- \lambda_2 \biggr) \biggr\}
\end{eqnarray*}
with parameter range $a,\sigma\in\mathbb{R},   c_1, c_2, \alpha_1,
\alpha_2, \beta_1, \beta_2>0$ and $\lambda_1, \lambda_2\in
(0,3)\setminus\{1,2\}$.
Here ${\mathrm B}(x,y)=\Gamma(x)\Gamma(y)/\Gamma(x+y)$ is the
Beta function (see \cite{Jef2007}).
The density of the L\'evy measure is given by
\[
\pi(x) = c_1\frac{e^{-\alpha_1 \beta_1 x}}{(1- e^{-\beta_1
x})^{\lambda_1}}\mathbf{1}_{\{x>0\}}
+ c_2 \frac{e^{\alpha_2\beta_2 x}}{(1- e^{\beta_2 x})^{\lambda
_2}}\mathbf{1}_{\{x<0\}}.
\]
Although $\Psi$ takes a seemingly complicated form, this particular
family of L\'evy processes has a number of very beneficial virtues from
the point of view of mathematical finance which are discussed in \cite{Kuz}.
Moreover, the large number of parameters also allows one to choose L\'
evy processes within the $\beta$-class that have paths that are both
of unbounded variation [when at least one of the conditions $\sigma
\neq0$, $\lambda_1\in(2,3)$ or $\lambda_2\in(2,3)$ holds] and
bounded variation [when all of the conditions $\sigma=0$, $\lambda
_1\in(0,2)$ and $\lambda_2\in(0,2)$ hold] as well as having infinite
and finite activity in the jumps component [accordingly as both
$\lambda_1,\lambda_2\in(1,3)$ or not].

What is special about the $\beta$-class is that all the roots of the
equation $\lambda+\Psi(\theta)=0$ are analytically identifiable
which leads to semi-explicit identities for the laws of $\overline
{X}_{\mathbf{e}_1/\lambda}$ and $-\underline{X}_{\mathbf
{e}_1/\lambda}$ as the following result lifted from \cite{Kuz} shows.
\begin{theorem}\label{kuz1}
For $\lambda>0$, all the roots of the equation
\[
\lambda+\Psi(\theta) = 0
\]
are simple and occur on the imaginary axis. They can be enumerated by
$\{\mathrm{i}\zeta_n^+\dvtx  n\geq0\}$ on the positive imaginary axis and $\{
\mathrm{i}\zeta_n^- \dvtx  n\geq0\}$ on the negative imaginary axis in order
of increasing absolute magnitude where
\begin{eqnarray*}
 \zeta^+_0 &\in&(0, \beta_2\alpha_2),  \qquad   \zeta^-_{0} \in(-\beta
_1\alpha_1 , 0),  \\
 \zeta^+_n &\in&\bigl(\beta_2(\alpha_2 +n-1),   \beta_2(\alpha_2+n)\bigr) \qquad \mbox
{for }n\geq1, \\
 \zeta^-_{n} &\in&\bigl(\beta_1(-\alpha_1 -n),   \beta_1(-\alpha
_1-n+1)\bigr) \qquad \mbox{for }n\geq1.
\end{eqnarray*}
Moreover, for $x>0$,
%
\begin{equation}
\mathbb{P}(\overline{X}_{\mathbf{e}_1 /\lambda}\in\mathrm{d}x) =
-\biggl (\sum_{k\geq0} c_k^- \zeta_k^- e^{\zeta^-_k x}
\biggr)\,\mathrm{d}x,
\label{max-density}
\end{equation}
where
\[
c_0^- = \prod_{n\geq1} \frac{1+ {\zeta_0^-}/({\beta
_1(n-1+\alpha_1)})}{1-  {\zeta_0^-}/{\zeta_n^-}}
\]
and
\[
c_k^- =
\frac{1+  {\zeta_k^-}/({\beta_1(k-1+\alpha_1)})}{1-  {\zeta
_k^-}/{\zeta_0^-}}
\prod_{n\geq1, n\neq k} \frac{1+ {\zeta_k^-}/({\beta
_1(n-1+\alpha_1)})}{1-  {\zeta_k^-}/{\zeta_n^-}}.
\]
A similar expression holds for $\mathbb{P}(-\underline{X}_{\mathbf
{e}_1 /\lambda}\in\mathrm{d}x)$ with the role of $\{\zeta^-_n\dvtx  n\geq0\}
$ being played by $\{- \zeta^+_n \dvtx  n\geq0\}$ and $\alpha_1, \beta
_1$ replaced by $\alpha_2, \beta_2$.
\end{theorem}

Note that when $0$ is irregular for $(0,\infty)$ the distribution of
$\overline{X}_{\mathbf{e}_1/\lambda}$ will have an atom at $0$ which
can be computed from (\ref{max-density}) and is equal to $1-\sum
_{k\geq0} c_k^-$. Alternatively, from Remark 6 in \cite{Kuz} this can
equivalently be written as
$\prod_{n\geq0} (-\zeta^-_n) /\beta_1(n+\alpha_1)$. A similar
statement can be made concerning an atom at $0$ for the distribution of
$-\underline{X}_{\mathbf{e}_1/\lambda}$ when $0$ is irregular for
$(-\infty,0)$. Conditions for irregularity are easy to check thanks to
Bertoin \cite{irreg}; see also the summary in Kyprianou and Loeffen
\cite{KL} for other types of L\'evy processes that are popular in
mathematical finance.

By making a suitable truncation of the series (\ref{max-density}), one
may easily perform independent sampling from the distributions
$\overline{X}_{\mathbf{e}_1 /\lambda}$ and $\underline{X}_{\mathbf
{e}_1 /\lambda}$ as required for our Monte Carlo methods.

\subsection{Philanthropy and general hypergeometric L\'evy processes}

The forthcoming discussion will assume familiarity with classical
excursion theory of L\'evy processes for which the reader is referred
to Chapter VI of \cite{Bertbook} or Chapter 6 of \cite{Kypbook}.

According to Vigon's theory of philanthropy, a (killed) subordinator is
called a \textit{philanthropist} if its L\'evy measure has a decreasing
density on $\R_+$. Moreover, given any two subordinators $H_1$ and
${H_2}$ which are philanthropists, providing that at least one of them
is not killed, there exists a L\'evy process~$X$ such that $H_1$ and
${H_2}$ have the same law as the ascending and descending ladder height
processes of $X$, respectively. (In the language of Vigon, the
philanthropists $H_1$ and $H_2$ are \textit{friends}.)
Suppose we denote the killing rate, drift coefficient and L\'evy
measures of $H_1$ and ${H_2}$ by the respective triples $(k,\delta,\Pi
_{H_1})$ and $(\widehat{k},\widehat{\delta},\Pi_{{H_2}})$.
Then \cite{V} shows that the L\'evy measure of $X$ satisfies the
following identity:
%
\begin{equation}\label{phillm}
 \quad \overlineh{\Pi}^+_X(x)=\int_0^{\infty}\Pi_{H_1}(x+\mathrm{d} u)\overlineh
{\Pi}_{{H_2}}(u)+\widehat{\delta} \pi_{H_1}(x)+\widehat{k}^{\,}
\overlineh{\Pi}_{H_1}(x), \qquad x>0,
\end{equation}
where $\overlineh{\Pi}^+_X(x) := \Pi_X(x,\infty)$, $\overlineh{\Pi
}_{{H_1}}(u): = \Pi_{{H_1}}(u,\infty)$, $\overlineh{\Pi}_{{H_2}}(u):
= \Pi_{{H_2}}(u,\infty)$\break and
$\pi_{H_1}$ is the density of $\Pi_{H_1}$.
By symmetry, an obvious analogue of (\ref{phillm}) holds for the
negative tail $\overlineh{\Pi}_X^-(x): = \Pi_X(-\infty, x)$, $x<0$.

A particular family of subordinators which will be of interest to us is
the class of subordinators which is found within the definition of
Kuznetsov's $\beta$-class of L\'evy processes. These processes have
characteristics $(c, \alpha,\beta, \gamma)$ where $\gamma\in
(-\infty,0) \cup(0,1) $, $\beta, c>0$ and $1-\alpha+\gamma>0$. The
L\'evy measure of such subordinators is of the type
%
\begin{equation}\label{mlls}
c\frac{e^{\alpha\beta x}}{(e^{\beta x}-1)^{1+\gamma}}1_{\{x>0\}}\,
\mathrm{d}x.
\end{equation}

From Proposition 9 in \cite{Kuz}, the Laplace exponent of a $\beta
$-class subordinator satisfies
\begin{equation}\label{kuz}
\Phi(\theta)
=\mathtt{k}+\delta\theta+\frac{c}{\beta}  \{
{\mathrm{B}} (1-\alpha+\gamma, - \gamma ) -
{\mathrm{B}} ( 1-\alpha+\gamma+\theta/\beta, - \gamma
 ) \}
\end{equation}
for $\theta\geq0$
where $\delta$ is the drift coefficient and $\mathtt{k}$ is the
killing rate.

Let $H_1$ and $H_2$ be two independent subordinators from the $\beta
$-class where for $i=1,2,$ with respective drift coefficients $\delta
_i\ge0$, killing rates $\mathtt{k}_i\geq0$ and L\'evy measure
parameters $(c_i,\alpha_i,\beta,\gamma_i)$.
Their respective Laplace exponents are denoted by $\Phi_i$, $i=1,2$.
In Vigon's theory of philanthropy, it is required that $\mathtt
{k}_1\mathtt{k}_2=0$. Under this assumption, let us denote by $X$ the
L\'evy process whose ascending and descending ladder height processes
have the same law as $H_1$ and ${H_2}$, respectively. In other words,
the L\'evy process whose characteristic exponent is given by
$
\Phi_1(-\mathrm{i}\theta)\Phi_2(\mathrm{i}\theta),   \theta\in\mathbb{R}.
$
It is important to note that the Gaussian component of the process $X$
is given by $2\delta_1\delta_2$; see \cite{V}.
From (\ref{phillm}), the L\'evy measure of $X$ is such that
\begin{eqnarray*}
\overlineh{\Pi}^+_X(x)&=&c_1c_2\int_x^{\infty}\frac{e^{\beta
_1\alpha_1 u}}{(e^{\beta_1 u}-1)^{\gamma_1+1}}\int_{u-x}^{\infty
}\frac{e^{\alpha_2\beta_2 z}}{(e^{\beta_2z}-1)^{\gamma_2+1}}\,\mathrm{d}z\,
\mathrm{d}u\\
&&{}+\delta_2 c_1\frac{e^{\beta_1\alpha_1 x}}{(e^{\beta_1
x}-1)^{\gamma_1+1}}
+\mathtt{k}_2c_1\int_x^{\infty}\frac{e^{\beta_1\alpha_1
u}}{(e^{\beta_1 u}-1)^{\gamma_1+1}}\, \mathrm{d}x.
\end{eqnarray*}
Assume first that $\gamma_2<0$, taking derivative in $x$ and computing
the resulting integrals with the help of \cite{Jef2007} we find that
for $x>0$ the density of the L\'evy measure is given by
\begin{eqnarray*}
\pi(x)&=&-\frac{c_1c_2}{\beta} {\mathrm B}(\rho,-\gamma_2)
e^{-\beta x (1+\gamma_1-\alpha_1)} {}_2F_1(1+\gamma_1,\rho;\rho
-\gamma_2;e^{-\beta x})\\
&&{}+c_1  \biggl({\mathrm k}_2 +\frac{c_2}{\beta} {\mathrm
B}(1+\gamma_2-\alpha_2,-\gamma_2)  \biggr) \frac{e^{\alpha_1 \beta
x}}{(e^{\beta x}-1)^{1+\gamma_1}}\\
&&{}-
\delta_2 c_1 \frac{\d}{\d x}  \biggl[ \frac{e^{\alpha_1 \beta
x}}{(e^{\beta x}-1)^{1+\gamma_1}}  \biggr],
\end{eqnarray*}
where $\rho=2+\gamma_1+\gamma_2-\alpha_1-\alpha_2$. The validity
of this formula is extended for $\gamma_2 \in(0,1)$ by analytical
continuation.
The corresponding expression for $x<0$ can be obtained by symmetry
considerations.

We define a General Hypergeometric process to be the 13 parameter L\'
evy process with characteristic exponent given in compact form
%
\begin{equation}\label{compact}
\Psi(\theta) = \mathtt{d}\mathrm{i}\theta+ \frac{1}{2}\sigma^2
\theta^2 + \Phi_1(-\mathrm{i}\theta)\Phi_2(\mathrm{i}\theta),  \qquad   \theta
\in\mathbb{R},
\end{equation}
where $\mathtt{d},\sigma\in\mathbb{R}$.
The two additional parameters $\mathtt{d}, \sigma$ are included
largely with applications in mathematical finance in mind. Without
these two additional parameters, it is difficult to disentangle the
Gaussian coefficient and the drift coefficients from parameters
appearing in the jump measure. Note that the Gaussian coefficient in
(\ref{compact}) is now $\sigma^2/2 + 2\delta_1\delta_2$. The
definition of General Hypergeometric L\'evy processes includes
previously defined Hypergeometric L\'evy processes in Kyprianou, Pardo
and Rivero
\cite{KPR}, Caballero, Pardo and P{\'e}rez \cite{CPP2} and Lamperti-stable L\'evy
processes in Caballero, Pardo and P{\'e}rez~\cite{CPP}.\looseness=-1

Just as with the case of the $\beta$-family of L\'evy processes,
because $\Psi$ can be written as a linear combination of a quadratic
form and beta functions, it turns out that one can identify all the
roots of the equation $\Psi(\theta)+\lambda=0$ which is again
sufficient to describe the laws of $\overline{X}_{\mathbf
{e}_1/\lambda}$ and $-\underline{X}_{\mathbf{e}_1/\lambda}$.

\begin{theorem}\label{HG}
For $\lambda>0$, all the roots of the equation
\[
\lambda+\Psi(\theta) = 0
\]
are simple and occur on the imaginary axis.  They can be
enumerated by $\{\mathrm{i}\xi_n^+\dvtx  n\geq0\}$ on the positive imaginary
axis and $\{\mathrm{i}\xi_n^- \dvtx  n\geq0\}$ on the negative imaginary axis
in order of increasing absolute magnitude where
\begin{eqnarray*}
 \xi^+_0&\in& \bigl(0,\beta(1+\gamma_2-\alpha_2) \bigr),  \qquad  \xi^-_0\in
 \bigl(-\beta(1+\gamma_1-\alpha_1),0 \bigr), \\
 \xi^+_n&\in& \bigl(\beta(\gamma_2-\alpha_2 +n),\beta(1+\gamma
_2-\alpha_2 +n) \bigr)
 \qquad \mbox{for }n\geq1, \\
 \xi^-_n&\in& \bigl(-\beta(1+\gamma_1-\alpha_1+n),-\beta(\gamma
_1-\alpha_1+n) \bigr) \qquad \mbox{for }n\geq1.
\end{eqnarray*}
Moreover, for $x>0$,
%
\begin{equation}
\mathbb{P}(\overline{X}_{\mathbf{e}_1/\lambda}\in\mathrm{d} x)=
- \biggl(\sum_{k\ge0}c^-_k\xi^-_{k}e^{{\xi^-_k}x} \biggr)\,\mathrm{d}x,
\label{hypsup}
\end{equation}
where
\[
c^-_0=\prod_{n\ge1}\frac{1+ {\xi_0^-}/({\beta(\gamma_1-\alpha
_1+n)})}{1- {\xi_0^-}/{\xi^-_n}}
\]
and
\[
c^-_k=\frac{1+ {\xi_k^-}/({\beta(\gamma_1-\alpha_1+k)})}{1- {\xi_k^-}/{\xi^-_0}}\prod_{n\ge1, n\ne k}\frac{1+ {\xi
_k^-}/({\beta(\gamma_1-\alpha_1+n)})}{1- {\xi_k^-}/{\xi^-_n}}.
\]
 A similar expression holds for $\mathbb{P}(-\underline
{X}_{\mathbf{e}_1/\lambda}\in\mathrm{d}x)$ with the role of $\{\xi
^-_n\dvtx  n\geq0\}$ replaced by $\{-\xi^+_n\dvtx n\geq0\}$ and $\alpha_1,
\gamma_2$ replaced by $\alpha_2, \gamma_2$.
\end{theorem}

\begin{pf} The proof is very similar to the proof of Theorem 10 in
\cite{Kuz}. Formula (\ref{compact}) and reflection formula for the
Beta function (see \cite{Jef2007})
\begin{equation}\label{eq_beta_reflection}
{\mathrm B}(-z;-\gamma)={\mathrm B}(1+z+\gamma;-\gamma)   \frac
{\sin(\pi(z+\gamma))}{\sin(\pi z)}\vadjust{\goodbreak}
\end{equation}
tell us that $\Psi(\i\theta) \to-\infty$ as $\theta\to\beta
(1+\gamma_2 - \alpha_2)$, and since $\Psi(0)=0$ we conclude that
$\lambda+\Psi(\i\theta)=0$
has a solution on the interval $\theta\in(0,\beta(1+\gamma_2 -
\alpha_2))$. Other intervals can be checked in a similar way [note
that $\Phi_i(z)$ are Laplace exponents of subordinators, therefore
they are positive for \mbox{$z>0$}]. Next, we assume that $\sigma, \delta_1,
\delta_2>0$. Using formulas
(\ref{compact}), (\ref{eq_beta_reflection}) and the asymptotic result
\[
\frac{\Gamma(a+z)}{\Gamma(z)}=z^{a}+O(z^{a-1}),  \qquad
z\to+\infty,
\]
which can be found in \cite{Jef2007}, we conclude that $\Psi(\i
\theta)$ has the following asymptotics as $\theta\to+\infty$:
\begin{eqnarray*}
\Psi(\i\theta)&=&-\frac12 (\sigma^2+2\delta_1 \delta_2) \theta
^2+O(\theta^{1+\gamma_2})\\
&&{}-\frac{\delta_1\Gamma(-\gamma_2)}{\beta^{\gamma_2}}
\frac{\sin (\pi (\alpha_2+ {\theta}/{\beta}
)  )}{\sin (\pi (\alpha_2-\gamma_2+{\theta
}/{\beta}  )  )}
 [\theta^{1+\gamma_2} + O(\theta^{\gamma_2+\gamma_1})  ].
\end{eqnarray*}
Using the above asymptotic expansion and the same technique as in the
proof of Theorem 5 in \cite{Kuz}, we find that as $n\to+\infty$
there exists a constant~$C_1$ such that
\[
\xi_n^+ = \beta(n+1+\gamma_2-\alpha_2) + C_1 n^{\gamma_2-1} +
O(n^{\gamma_2-1-\epsilon}),
\]
with a similar expression for $\xi_n^-$. Thus, we use Lemma 6 from
\cite{Kuz} (and the same argument as in the proofs of Theorems 5 and
10 in \cite{Kuz}) to show
that first there exist no other roots of meromorphic function $\lambda
+\Psi(\i z)$ except for~$\{\xi_n^{\pm}\}$, and secondly that we have
a factorization
\begin{eqnarray*}
\frac{\lambda}{\lambda+\Psi(\theta)}&=&\frac{1}{1+ {\mathrm{i}\theta}/{\xi^-_0}}\prod_{n\ge1}\frac{1- {\mathrm{i}\theta
}/({\beta(\gamma_1-\alpha_1+n)})}{1+ {\mathrm{i}\theta}/{\xi^-_n}}\\
&&{}\times\frac{1}{1+ {\mathrm{i}\theta}/{\xi^+_0}}\prod_{n\ge
1}\frac{1+ {\mathrm{i}\theta}/({\beta(\gamma_2-\alpha
_2+n)})}{1+ {\mathrm{i}\theta}/{\xi^+_n}}.
\end{eqnarray*}
The Wiener--Hopf factoris $\phi_q^{\pm}(\theta)$ are identified from
the above equation with the help of analytical uniqueness result, Lemma
2 in \cite{Kuz}.
Formula (\ref{hypsup}) is obtained from the infinite product
representation for $\phi_q^{+}(\theta)$ using residue calculus.

\begin{table}[t]
\tabcolsep=0pt
\caption{Coefficients for the asymptotic expansion of $\xi_n^{+}$}
\label{table2}
\begin{tabular*}{\textwidth}{@{\extracolsep{\fill}}lccc@{}}
\hline
\textbf{Case} & $\bolds{\omega_2}$ & $\bolds{C}$ & $\bolds{\varrho_2}$\\
\hline
$\sigma^2, \delta_1,\delta_2>0$ & $1+\gamma_2$ & $\frac{2\delta
_1c_2}{\beta\Gamma(1+\gamma_2)(\sigma^2+2\delta_1\delta_2)}$ &
$\gamma_2-1$\\[4pt]
$\sigma=0, \delta_1,\delta_2>0$ & $1+\gamma_2$ & $\frac{c_2}{\beta
\Gamma(1+\gamma_2)\delta_2}$ & $\gamma_2-1$\\[4pt]
$\sigma^2, \delta_2>0, \delta_1=0$ & $1+\gamma_2$ & $\frac
{2c_1c_2\Gamma(1-\gamma_1)}{\beta^{3+\gamma_1-\gamma_2}\Gamma
(1+\gamma_2)\gamma_1\sigma^2}$ & $\gamma_1+\gamma_2-2$\\[2pt]
$\sigma^2, \delta_1>0, \delta_2=0$ & $1+\gamma_2$ & $\frac{2\delta
_1c_2}{\beta\Gamma(1+\gamma_2)\sigma^2}$ & $\gamma_2-1$ \\[4pt]
$\delta_2>0, \sigma=\delta_1=0$ & $1+\gamma_2$ & $\frac{c_2}{\beta
\delta_2\Gamma(1+\gamma_2)}$ & $\gamma_2-1$\\[3pt]
$\delta_1>0, \sigma=\delta_2=0$ & $0$ & $\frac{\sin(\pi\gamma
_2)}{\pi}\frac{\beta^2\gamma_2(\mu+\mathtt{d})}{\delta
_1c_2\Gamma(1-\gamma_2)}$ & $-\gamma_2$ \\[3pt]
$\sigma^2>0, \delta_1=\delta_2=0$ & $1+\gamma_2$ & $\frac
{2c_1c_2\Gamma(1-\gamma_1)}{\beta^{3+\gamma_1-\gamma_2}\Gamma
(1+\gamma_2)\gamma_1\sigma^2}$ & $\gamma_1+\gamma_2-2$ \\
$\sigma=\delta_1=\delta_2=0$ & $1$ &
$\frac{\beta^2\gamma_2}{c_2\Gamma(1-\gamma_2)}\frac{\sin(\pi
\gamma_2)}{\pi} (\mathtt{k}_2+\frac{c_2}{\beta}{\mathrm
B}(1+\gamma_2-\alpha_2;-\gamma_2) ) $
& $-\gamma_2$ \\
\hline
\end{tabular*}
\end{table}

This ends the proof in the case $\sigma, \delta_1, \delta_2>0$, in
all other cases the proof is almost identical, except that one has to do
more work to obtain asymptotics for the roots of $\lambda+\Psi(\i
\theta)=0$. We summarize all the possible asymptotics of the roots below
\[
\xi^+_n= \beta(n-\alpha_2+\omega_2)+C n^{\varrho_2}+O(n^{\varrho
_2-\epsilon})\qquad\mbox{as }  n\to\infty,
\]
where the coefficients $\omega_2, \varrho_2$ and $C$ are presented in
Table \ref{table2}. Corresponding results for $\xi_n^-$ can be
obtained by symmetry considerations.
\end{pf}

\begin{remark} Similar comments to those made after Theorem \ref
{kuz1} regarding the existence of atoms in the distribution of
$\overline{X}_{\mathbf{e}_1/\lambda}$ and $-\underline{X}_{\mathbf
{e}_1/\lambda}$ also apply here.
\end{remark}

\begin{remark}
It is important to note that the hypergeometric L\'evy process is but
one of many examples of L\'evy processes which may be constructed using
Vigon's theory of philanthropy. With the current Monte Carlo algorithm
in mind, it should be possible to engineer other favorable L\'evy
processes in this way.
\end{remark}

\section{Extensions}\label{Sect:extensions}
\subsection{Building in arbitrary large jumps}

The starting point for the Wiener--Hopf Monte Carlo algorithm is the
distribution of $\overline{X}_{\mathbf{e}_1/\lambda}$
and $\underline{X}_{\mathbf{e}_1/\lambda}$, and in Section \ref
{sec_impl} we have presented two large families of L\'evy processes for
which one can compute these
distributions quite efficiently. We have also argued the case that one
might engineer other fit-for-purpose Wiener--Hopf factorizations using
Vigon's theory of philanthropy. However, below, we present another
alternative for extending the  application of the Wiener--Hopf
Monte Carlo technique to a much larger class of L\'evy processes than
those for which sufficient knowledge of the Wiener--Hopf factorization
is known.
Indeed the importance of Theorem \ref{thm_cmpnd_Poisson} below is that
we may now work with any L\'evy processes whose L\'evy measure can be
written as a sum of
a L\'evy measure from the $\beta$-family or hypergeometric family plus
\textit{any} other measure with finite mass. This is a very general class
as a little thought reveals that many L\'evy processes necessarily take
this form. However, there are some obvious exclusions from this class,
for example, cases of L\'evy processes with bounded jumps.

\begin{theorem}\label{thm_cmpnd_Poisson} Let $Y = \{Y_t\dvtx  t\geq0\}$ be
a sum of a L\'evy process $X$ and a~compound Poisson process such that
for all $t\geq0$,
\[
Y_t=X_t+\sum_{i=1}^{N_{t}} \xi_i,
\]
where $N = \{N_t\dvtx  t\geq0\}$ is a Poisson process with intensity
$\gamma$, independent of the i.i.d. sequence of random variables, $\{
\xi_i\dvtx  i\ge1\}$, and $X$. Define iteratively for $n\ge1$
\begin{eqnarray*}
V(n,\lambda) &=& V(n-1,\lambda)+S^{(n)}_{\lambda+\gamma}+
I^{(n)}_{\lambda+\gamma}+ \xi_n (1-\beta_n), \\
J(n, \lambda) &=& \max \bigl( V(n,\lambda)  , J(n-1,\lambda)  ,
V(n-1,\lambda)+ S^{(n)}_{\lambda+\gamma}  \bigr),
\end{eqnarray*}
where $V(0,\lambda)=J(0,\lambda)=0$, sequences $\{S^{(j)}_{\lambda
+\gamma}\dvtx  n\ge1\}$ and $\{I^{(n)}_{\lambda+\gamma}\dvtx  n\geq1\}$ are
defined in Theorem \ref{thm_main}, and
$\{\beta_n\dvtx  n\ge1\}$ are an i.i.d. sequence of Bernoulli random
variables such that
$\p(\beta_n=1)=\lambda/(\gamma+\lambda)$.
Then
\begin{equation}\label{distr_identity2}
\bigl(Y_{\mathbf{g}(n,\lambda)}, \overline{Y}_{\mathbf{g}(n,\lambda)}\bigr)
\stackrel{d}{=}(V(T_n,\lambda), J(T_n,\lambda)),
\end{equation}
where
$ T_n=\min\{ j\ge1   \dvtx    \sum_{i=1}^{j} \beta_i=n \}$.
\end{theorem}

\begin{pf} Consider a Poisson process with arrival rate $\lambda
+\gamma$ such that points are independently marked with probability
$\lambda/(\lambda+\gamma)$. Then recall that the Poisson Thinning
theorem tells us that the process of marked points is a Poisson process
with arrival rate $\lambda$. In particular, the
arrival time having index $T_1$ is exponentially distributed with rate
$\lambda$.

Suppose that $\tau_1$ is the first time that an arrival occurs in the
process $N$, in particular $\tau_1$ is exponentially distributed with
rate $\gamma$. Let $\mathbf{e}_\lambda$ be another independent and
exponentially distributed random variable, and fix $x\in\mathbb{R}$
and $y\geq0$. Then making use of the Wiener--Hopf decomposition,
\begin{eqnarray*}
 &&(x +Y_{\tau_1\wedge\mathbf{e}_\lambda},\max\{ y, x
+\overline{Y}_ {\tau_1\wedge\mathbf{e}_\lambda}\}) \\
&& \qquad =
\cases{\displaystyle
\bigl(x +S^{(1)}_{\lambda}+ I^{(1)}_{\lambda},   \max\bigl\{ y,   x
+S^{(1)}_\lambda\bigr\}\bigr) ,  \qquad   \mbox{if } \mathbf{e}_\lambda< \tau_1,\vspace*{2pt}\cr\displaystyle
\bigl(x+ S^{(1)}_{\gamma}+ I^{(1)}_{\gamma} + \xi_n ,  \max\bigl\{
x+S^{(1)}_{\gamma}+ I^{(1)}_{\gamma} + \xi_n,   y,  x +
S^{(1)}_\gamma\bigr\}\bigr) ,\cr
 \hspace*{161pt}\qquad   \mbox{if }   \tau_1\leq\mathbf{e}_\lambda.
}
\end{eqnarray*}
If we momentarily set $(x,y) = (V(0,\lambda), J(0,\lambda)) = (0,0)$,
then by the Poisson Thinning theorem it follows that $(Y_{\tau_1\wedge
\mathbf{e}_\lambda}, \overline{Y}_ {\tau_1\wedge\mathbf
{e}_\lambda})$ is equal in distribution to $(V(1,\lambda),
J(1,\lambda))$. Moreover, again by the Poisson Thinning theorem, $(Y_{
\mathbf{e}_\lambda}, \overline{Y}_ { \mathbf{e}_\lambda})$ is
equal in distribution to $(V(T_1,\lambda), J(T_1,\lambda))$. This
proves the theorem for the case $n=1$.

In the spirit of the proof of Theorem \ref{thm_main}, the proof for
$n\geq2$ can be established by an inductive argument. Indeed, if the
result is true for $n=k-1$ then it is true for $n=k$ by taking $(x,y) =
(V(k-1, \lambda), J(k-1, \lambda))$ then appealing to the lack of
memory property, stationary and independent increments of $Y$ and the
above analysis for the case that $n=1$. The details are left to the reader.
\end{pf}

\begin{remark}
A particular example where the use of the above theorem is of
pertinence is a linear Brownian motion plus an independent compound
Poisson process. This would include, for example, the so-called Kou model
from mathematical finance in which the jumps of the compound Poisson
process have a two-sided exponential distribution. In the case that $X$
is a linear Brownian motion, the quantities $\overline{X}_{\mathbf
{e}_1/\lambda}$ and $-\underline{X}_{\mathbf{e}_1/\lambda}$ are
both exponentially distributed with easily computed rates.
\end{remark}

\subsection{Approximate simulation of the law of $(X_t, \overline
{X}_t, \underline{X}_t)$}

Next, we consider the problem of sampling from the distribution of the
three random variables $(X_t, \overline{X}_t, \underline{X}_t)$. This
is also an important problem for applications making use of the
two-sided exit problem and, in particular, for pricing double barrier options.
The following slight modification of the Wiener--Hopf Monte Carlo
technique allows us to obtain two estimates for this triple of random variables,
which in many cases can be used to provide upper and lower bounds for
certain functionals of $(X_t, \overline{X}_t, \underline{X}_t)$.

\begin{theorem}\label{thm_3d_distr}
Given two sequences $\{S^{(n)}_\lambda\dvtx  n\ge1\}$ and $\{
I^{(n)}_\lambda\dvtx  n\ge1\}$ introduced in Theorem \ref{thm_main} we define
iteratively for $n\ge1$
\begin{eqnarray}\label{def_5_rand_vrls}
V(n,\lambda) &=& V(n-1,\lambda)+S^{(n)}_\lambda+ I^{(n)}_\lambda,
\nonumber\\
J(n, \lambda) &=& \max \bigl( J(n-1,\lambda), V(n-1,\lambda)+
S^{(n)}_\lambda \bigr),
\nonumber\\
K(n,\lambda) &=& \min \bigl( K(n-1,\lambda), V(n,\lambda)  \bigr),
\\
\tilde J(n, \lambda) &=& \max \bigl( \tilde J (n-1,\lambda),
V(n,\lambda)  \bigr), \nonumber\\
\tilde K(n,\lambda) &=& \min \bigl( \tilde K(n-1,\lambda),
V(n-1,\lambda)+ I^{(n)}_\lambda \bigr),\nonumber
\end{eqnarray}
where $V(0,\lambda)=J(0,\lambda)=K(0,\lambda)=\tilde J(0,\lambda
)=\tilde K(0,\lambda)=0$. Then for any bounded function $f(x,y,z)\dvtx  \r
^3 \mapsto\r$ which is increasing in
$z$-variable we have\looseness=1
\begin{eqnarray}\label{bias-1}
\e[ f(V(n,\lambda), J(n, \lambda) , K(n,\lambda)) ] &\ge& \e
\bigl[f\bigl(X_{\mathbf{g}(n,\lambda)},\overline{X}_{\mathbf{g}(n,\lambda
)},\underline{X}_{\mathbf{g}(n,\lambda)}\bigr)\bigr], \\\label{bias-2}
\e[ f(V(n,\lambda), \tilde K(n,\lambda), \tilde J(n, \lambda)) ]
&\le& \e\bigl[f\bigl(X_{\mathbf{g}(n,\lambda)},\underline{X}_{\mathbf
{g}(n,\lambda)},
\overline{X}_{\mathbf{g}(n,\lambda)}\bigr)\bigr].
\end{eqnarray}\looseness=0
\end{theorem}

\begin{pf} From Theorem \ref{thm_main}, we know that $(V(n,\lambda
), J(n,\lambda))$ has the same distribution as $(X_{\mathbf
{g}(n,\lambda)},\overline{X}_{\mathbf{g}(n,\lambda)})$, and, for
each $n\geq1$, $K(n,\lambda) =
\min\{X_{\mathbf{g}(k,\lambda)}\dvtx\allowbreak
k=0,1,\ldots,  n\}\geq\underline{X}_{\mathbf{g}(n,\lambda)}$. The
inequality in (\ref{bias-1}) now follows. The equality in (\ref
{bias-2}) is the result of a similar argument where now, for each
$n\geq1$, $\tilde{K}(n, \lambda) = \underline{X}_{\mathbf
{g}(n,\lambda)}$ and $\tilde{J}(n,\lambda) = \max\{X_{\mathbf
{g}(k, \lambda)}\dvtx  k = 0,1,\ldots , n\}\leq\overline{X}_{\mathbf
{g}(n,\lambda)}$.\vadjust{\goodbreak}
\end{pf}

Theorem \ref{thm_3d_distr} can be understood in the following sense.
Both triples of random variables $(V(n,\lambda), J(n,\lambda),
K(n,\lambda))$ and
$(V(n,\lambda), \tilde J(n,\lambda), \tilde K(n,\lambda))$ can be
considered as estimates for $(X_{\mathbf{g}(n,\lambda)},\overline
{X}_{\mathbf{g}(n,\lambda)},\underline{X}_{\mathbf{g}(n,\lambda
)})$, where in the first case
$K(n,\lambda)$ has a positive bias and in the second case $\tilde
J(n,\lambda)$ has a negative bias. An example of this is handled in
the next section.

\section{Numerical results}

In this section, we present numerical results. We perform computations
for a process $X_t$ in the $\beta$-family with parameters
\[
(a,\sigma, \alpha_1,\beta_1,\lambda_1, c_1, \alpha_2,\beta
_2,\lambda_2, c_2) = (a,\sigma, 1, 1.5, 1.5, 1, 1, 1.5, 1.5 , 1),
\]
where the linear drift $a$ is chosen such that $\Psi(-\mathrm{i})=-r$
with $r=0.05$, for no other reason that this is a risk neutral setting
which makes the process
$\{\exp(X_t-rt) \dvtx  t \geq0\}$ a martingale. We are interested in two
parameter sets. Set 1 has $\sigma=0.4$ and Set~2 has $\sigma=0$. Note
that both parameter
sets give us proceses with jumps of infinite activity but of bounded
variation, but due to the presence of Gaussian component the process
$X_t$ has unbounded variation
in the case of parameter Set 1.

As the first example, we compare computations of the joint density of
$(\overline{X}_1,\overline{X}_1-{X}_1)$ for the parameter Set 1.
Our first method is based on the following Fourier inversion technique.
As in the proof of Theorem \ref{thm_main}, we use the fact
that $\overline{X}_{\mathbf{e}_1/\lambda}$ and $X_{\mathbf
{e}_1/\lambda} -\overline{X}_{\mathbf{e}_1/\lambda}$ are independent,
and the latter is equal in distribution to $\underline{X}_{\mathbf
{e}_1/\lambda}$, to write
\begin{eqnarray*}
\p(\overline{X}_{\mathbf{e}_1/\lambda} \in\d x) \p(-\underline
{X}_{\mathbf{e}_1/\lambda} \in\d y) &=&
\p(\overline{X}_{\mathbf{e}_1/\lambda} \in\d x,   \overline
{X}_{\mathbf{e}_1/\lambda}-{X}_{\mathbf{e}_1/\lambda} \in\d y)\\
&=& \lambda\int_{\r^+} e^{-\lambda t} \p(\overline{X}_t
\in\d x,\overline{X}_t-{X}_t \in\d y) \,\d t.
\end{eqnarray*}
Writing down the inverse Laplace transform, we obtain
\begin{eqnarray}\label{eq_Fourier}
&&\p(\overline{X}_t \in\d x,\overline{X}_t-{X}_t \in\d y)\nonumber
\\[-8pt]
\\[-8pt]&& \qquad =\frac
{1}{2\pi\i} \int_{\lambda_0+\i\r}
\p(\overline{X}_{\mathbf{e}_1/\lambda} \in\d x) \p(-\underline
{X}_{\mathbf{e}_1/\lambda} \in\d y) \lambda^{-1} e^{\lambda t} \,\d
\lambda,
\nonumber
\end{eqnarray}
where $\lambda_0$ is any positive number. The values of analytical
continuation of $\p(\overline{X}_{\mathbf{e}_1/\lambda} \in\d x)$
for complex values of $\lambda$ can be computed efficiently using
technique described in \cite{Kuz}. Our numerical results indicate that
the integral in (\ref{eq_Fourier})
can be computed very precisely, provided that we use a large number of
discretization points in $\lambda$ space coupled with Filon-type
method to compute this Fourier type integral. Thus, first we compute
the joint density of $(\overline{X}_1,\overline{X}_1-{X}_1)$ using
(\ref{eq_Fourier}) and take it as a benchmark, which we use later
to compare
the Wiener--Hopf Monte Carlo method and the classical Monte Carlo
approach. For both of these methods, we fix the number of simulations
$M=10^7$ and the number of
time steps $N \in\{20,50,100\}$. For fair comparison, we use $2N$ time
steps for the classical Monte Carlo, as Wiener--Hopf Monte Carlo method\vadjust{\goodbreak}
with $N$ time steps
requires simulation of $2N$ random variables $\{S^{(j)}_{\lambda
},I^{(j)}_{\lambda} \dvtx  j=1,2,\ldots ,N\}$. All the code was written in
Fortran and the computations were performed on a
standard laptop (Intel Core 2 Duo~2.5 GHz processor and 3 GB of RAM).

%
\begin{figure}[t!]

\includegraphics{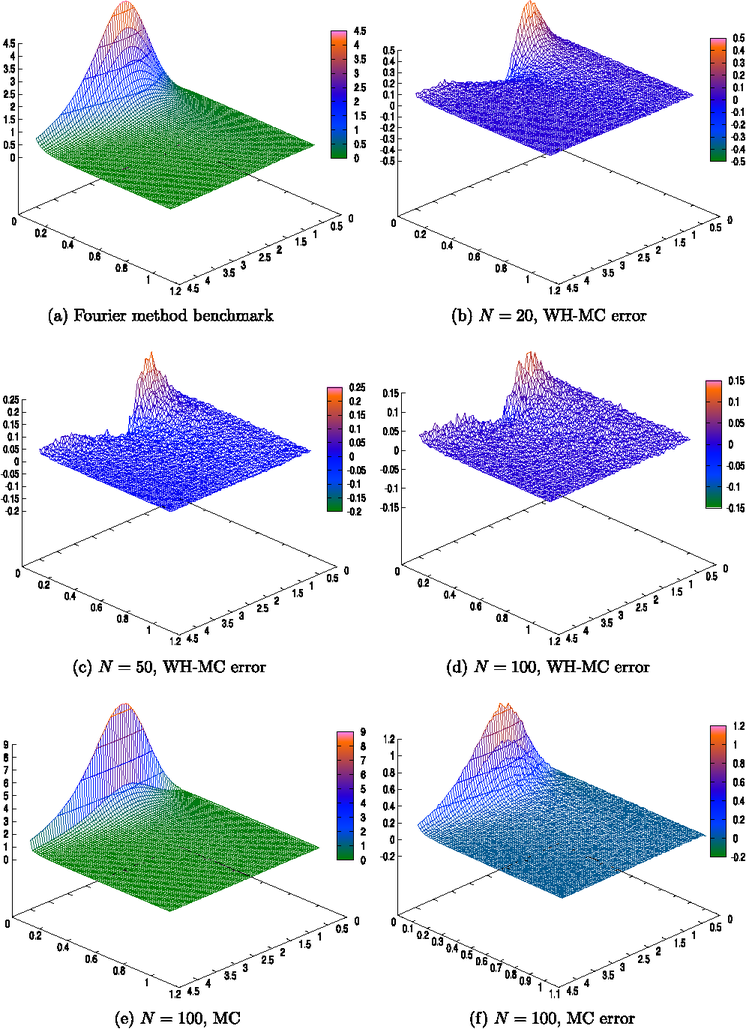}

\caption{Computing the joint density of $(\overline{X}_1,\overline
{X}_1-X_1)$ for parameter Set 1.
Here $\overline{X}_1 \in[0,1]$ and $\overline{X}_1-X_1 \in[0,4]$.}
\label{fig_2D_density}
\end{figure}

Figure \ref{fig_2D_density} presents the results of our computations.
In Figure \ref{fig_2D_density}(a), we show our benchmark, a surface
plot of the joint probability density
function of $(\overline{X}_1,\overline{X}_1-{X}_1)$ produced using
Fourier method (\ref{eq_Fourier}), which takes around 40--60 seconds
to compute.
Figure \ref{fig_2D_density}(b)--(d) show the difference between the benchmark and the
Wiener--Hopf Monte Carlo result
as the number of time steps $N$ increases from 20 to 50 to 100. The
computations take around 7 seconds for $N=100$, and 99\% of this time
is actually spent performing the Monte Carlo algorithm, as the
precomputations of the roots $\zeta_n^{\pm}$ and the law of
$I_{\lambda}, S_{\lambda}$ take less than one tenth of a second.
Figure \ref{fig_2D_density}(e) shows the result produced by the
classical Monte Carlo method with $N=100$ (which translates into 200
random walk steps according to our previous convention); this
computation takes around 10--15 seconds since here
we also need to compute the law of $X_{1/N}$, which is done using
inverse Fourier transform of the characteristic function of $X_t$ given
in (\ref{eq1}).
Finally, Figure \ref{fig_2D_density}(f) shows the difference between
the Monte Carlo result and our benchmark.

The results illustrate that in this particular example the Wiener--Hopf
Monte Carlo technique is superior to the classical Monte Carlo
approach. It gives a much more precise result, it requires less
computational time, is more straightforward to programme and does not
suffer from some the issues that plague the Monte Carlo approach, such
as the atom in distribution
of $\overline{X}_1$ at zero, which is clearly visible in Figure \ref
{fig_2D_density}(e).

%
\begin{figure}[t]

\includegraphics{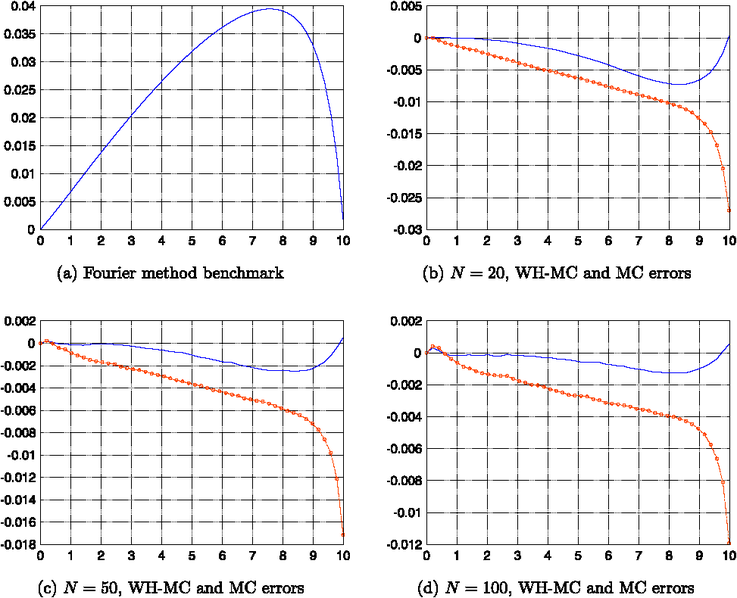}

\caption{Computing the price of up-and-out barrier option for
parameter Set 1. In figures \textup{(b)--(d)} the graph of WH-MC error
is
solid line, the graph of MC error
is   line with circles.}
\label{fig_barrier1}
\vspace*{2pt}
\end{figure}

Next, we consider the problem of pricing up-and-out barrier call option
with maturity equal to one, which is equivalent to computing the
following expectation:
%
\begin{equation}\label{barrier1}
\pi^{\mathrm{uo}}(s) = e^{-r} \mathbb{E}  \bigl[ (se^{X_1}-K)^+
\mathbf{1}_{\{s\exp(\overline{X}_1)< b \}}  \bigr].
\end{equation}
Here $s \in[0,b]$ is the initial stock price. We fix the strike price
$K=5$, the barrier level $b=10$. The numerical results for parameter
Set 1
are presented in Figure \ref{fig_barrier1}. Figure \ref
{fig_barrier1}(a) shows the graph of $\pi^{\mathrm{uo}}(s)$ as a
function of~$s$ produced with
Fourier method similar to (\ref{eq_Fourier}), which we again use as a~%
benchmark. Figure \ref{fig_barrier1}(b)--(d) show the difference between the benchmark and
results produced by Wiener--Hopf Monte Carlo (blue solid line) and
classical Monte Carlo (red line with circles) for $N\in\{20,50,100\}$.
Again we see
that Wiener--Hopf Monte Carlo method gives a better accuracy, especially
when the initial stock price level $s$ is close to the barrier $b$, as
in this case the Monte Carlo
approach produces an artificial atom in the distribution of~$\overline{X}_1$ at
zero which creates a large error.

%
\begin{figure}

\includegraphics{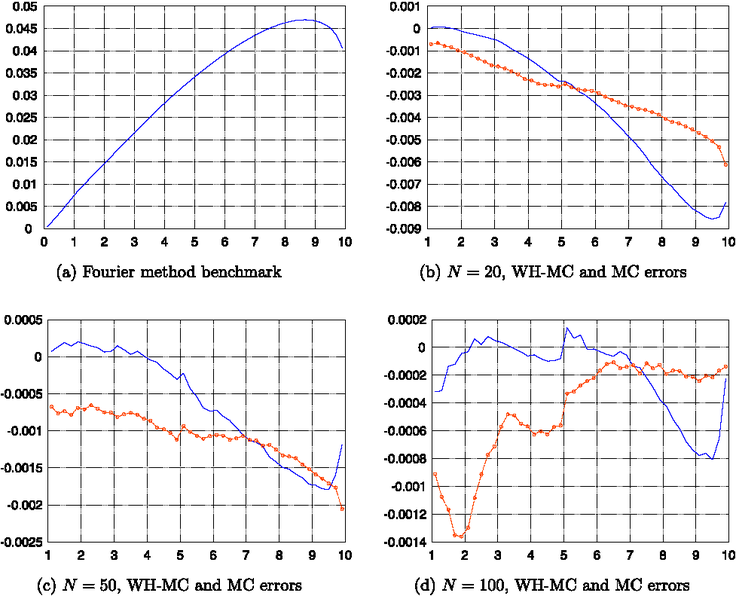}

\caption{Computing the price of up-and-out barrier option for
parameter Set 2. In figures~\textup{(b)--(d)} the graph of WH-MC error is
solid line, the graph of MC error
is   line with circles.}
\label{fig_barrier2}
\end{figure}

Figure \ref{fig_barrier2} shows corresponding numerical results for
parameter Set 2. In this case, we have an
interesting phenomenon of a discontinuity in $\pi^{\mathrm
{uo}}(s)$ at the boundary $b$.
The discontinuity should be there and occurs due to the fact\vadjust{\goodbreak} that, for
those particular parameter choices, there is irregularity of the upper
half line.
Irregularity of the upper half line is equivalent to there being an
atom at zero in the distribution of $\overline{X}_{t}$ for any $t>0$
(also at independent and exponentially distributed random times). We
see from the results presented in Figures \ref{fig_barrier1} and \ref
{fig_barrier2} that Wiener--Hopf Monte Carlo
method correctly captures this phenomenon; the atom at zero is produced
if and only if the upper half line is irregular, while the classical
Monte Carlo approach
always generates an atom. Also, analyzing Figure \ref
{fig_barrier2}(b)--(d), we
see that in this case classical
Monte Carlo algorithm is also doing a good job and it is hard to find a
winner. This is not surprising, as in the case of parameter Set 2 the
process $X_t$ has bounded
variation, thus the bias produced in monitoring for supremum only at
discrete times is smaller than in the case of process of unbounded
variation.

Finally, we give an example of how one can use Theorem \ref
{thm_3d_distr} to produce upper/lower bounds for the price of the
double no-touch barrier call option
%
\begin{equation}\label{barrier_dnt}
\pi^{\mathrm{dnt}}(s) = e^{-r} \mathbb{E}  \bigl[ (se^{X_1}-K)^+
\mathbf{1}_{\{s\exp(\overline{X}_1)< \overline{b}  ;   s\exp
(\underline{X}_1)>
\underline{b} \}} \bigr].
\end{equation}
First, we use identity $\mathbf{1}_{\{ s\exp(\underline{X}_1)>
\underline{b} \}}=1-\mathbf{1}_{\{s\exp(\underline{X}_1)<
\underline{b} \}}$ and obtain
\[
\pi^{\mathrm{dnt}}(s) = \pi^{\mathrm{uo}}(s) - e^{-r}
\mathbb{E}  \bigl[ (se^{X_1}-K)^+ \mathbf{1}_{\{s\exp(\overline
{X}_1)< \overline{b}  ;   s\exp(\underline{X}_1)< \underline{b} \}
} \bigr].
\]
Function $f(x,y,z)=-(se^{x}-K)^+ \mathbf{1}_{\{s\exp(y)< \overline
{b}  ;   s\exp(z)< \underline{b} \}}$ is increasing in both
variables $y$ and $z$, thus
using Theorem \ref{thm_3d_distr} we find that
\begin{eqnarray*}
\pi_1^{\mathrm{dnt}}(s) &=&
\pi^{\mathrm{uo}}(s)
- e^{-r} \e \bigl[\bigl(se^{V(n,n)}-K\bigr)^+ \mathbf{1}_{\{s\exp(\tilde
J(n,n))<\overline{b}   ;   s\exp(\tilde K(n,n))<\underline{b} \}
} \bigr], \\
\pi_2^{\mathrm{dnt}}(s) &=&\pi^{\mathrm{uo}}(s)
- e^{-r} \e\bigl [ \bigl(se^{V(n,n)}-K\bigr)^+ \mathbf{1}_{\{s\exp
(J(n,n))<\overline{b}   ;   s\exp(K(n,n))<\underline{b} \}} \bigr]
\end{eqnarray*}
are the lower/upper bounds for $\pi^{\mathrm{dnt}}(s)$. Figure
\ref{barrier_dnt1} illustrates this algorithm for parameter Set 1, the other
parameters being fixed at $K=5$, $\underline{b}=3$, $\overline{b}=10$
and the number of time steps $N=200$ (400 for the classical
Monte Carlo). We see that the Monte Carlo approach gives a price which
is almost always larger than the upper bound produced by the
Wiener--Hopf Monte Carlo algorithm. This is not surprising, as in the case
of Monte Carlo approach we would have positive (negative) bias in the
estimate of infimum (supremum), and given that the payoff of the double
no-touch barrier option
is increasing in infimum and decreasing in supremum this amplifies the bias.

%
\begin{figure}

\includegraphics{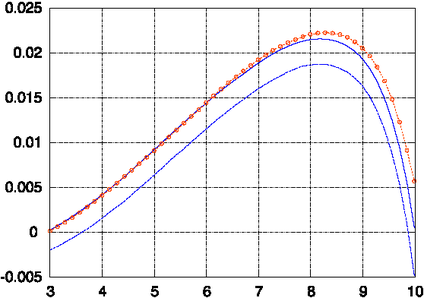}

\caption{Computing the price of the double no-touch barrier option for
parameter Set 1. The   solid lines represent the upper/lower bounds
produced by WH-MC method, the   line with circles represents the MC result.}
\label{barrier_dnt1}
\end{figure}

\section*{Acknowledgments} A. E. Kyprianou and K. van Schaik would also like to thank Alex
Cox for useful discussions.  All four
authors are grateful to two anonymous referees for their suggestions
which helped improved the presentation of this paper.

%

\printaddresses

\end{document}